\documentclass[dvipsnames,12pt]{amsart}

\usepackage[dvipsnames]{xcolor}
\usepackage{amssymb}
\usepackage{hyperref}
\usepackage{mathtools}
\usepackage{mathscinet}
\usepackage{tikz-cd}
\usepackage[normalem]{ulem}
\usepackage{tabstackengine}
	\setstackgap{L}{.75\baselineskip}
\usepackage{enumitem}
\usepackage[lmargin=3.5cm,rmargin=3.5cm,top=2.9cm,bottom=2.9cm]{geometry}

\newtheorem{theorem}{Theorem}[section]
\newtheorem{lemma}[theorem]{Lemma}
\newtheorem{proposition}[theorem]{Proposition}
\newtheorem{corollary}[theorem]{Corollary}

\theoremstyle{definition}
\newtheorem{definition}[theorem]{Definition}

\newtheorem{example}[theorem]{Example}

\theoremstyle{remark}
\newtheorem{remark}[theorem]{Remark}

\numberwithin{equation}{section}

\newcommand{\D}{\mathcal{D}}
\newcommand{\T}{\mathcal{T}}
\newcommand{\F}{\mathcal{F}}
\newcommand{\W}{\mathcal{W}}

\newcommand{\ip}[2]{\langle #1, #2 \rangle}
\newcommand{\sbm}[1]{{\let\amp=&\left(\begin{smallmatrix}#1\end{smallmatrix}\right)}}

\DeclareMathOperator{\Hom}{Hom}
\DeclareMathOperator{\mc}{mod}

\DeclareMathOperator{\dimu}{\underline{dim}}

\DeclareMathOperator{\Fac}{Fac}
\DeclareMathOperator{\Sub}{Sub}
\DeclareMathOperator{\End}{End}

\newcommand{\interior}[1]{%
  {\kern0pt#1}^{\mathrm{o}}%
}

\newcommand{\cat}[1]{\mathfrak{C}(#1)}
\newcommand{\tcm}[1]{\mathfrak{W}(#1)}
\newcommand{\tps}[2]{\mathcal{J}_{#1}(#2)}

\newcommand{\rep}[1]{%
  {%
    \tiny%
    \begin{matrix}%
      #1%
    \end{matrix}%
  }%
}

\newcommand{\np}[1]{\nu_{#1}}
\newcommand{\lspan}[1]{\mathrm{span}\{#1\}}
\newcommand{\strr}[2]{\mathrm{s}\tau\text{-}\mathrm{rigid}_{(#2)}#1}
\newcommand{\str}[1]{\mathrm{s}\tau\text{-}\mathrm{rigid}\,#1}

\newcommand{\trc}[1]{\mathfrak{T}(#1)}
\newcommand{\qtrc}[1]{\mathfrak{Q}(#1)}

\newcommand{\st}{:}

\makeatletter
\newcommand\@dotsep{4.5}
\def\@tocline#1#2#3#4#5#6#7{\relax
  \ifnum #1>\c@tocdepth 
  \else
    \par \addpenalty\@secpenalty\addvspace{#2}%
    \begingroup \hyphenpenalty\@M
    \@ifempty{#4}{%
      \@tempdima\csname r@tocindent\number#1\endcsname\relax
    }{%
      \@tempdima#4\relax
    }%
    \parindent\z@ \leftskip#3\relax \advance\leftskip\@tempdima\relax
    \rightskip\@pnumwidth plus1em \parfillskip-\@pnumwidth
    #5\leavevmode\hskip-\@tempdima{#6}\nobreak
    \leaders\hbox{$\m@th\mkern \@dotsep mu\hbox{.}\mkern \@dotsep mu$}\hfill
    \nobreak
    \hbox to\@pnumwidth{\@tocpagenum{\ifnum#1=1\fi#7}}\par%
    \nobreak
    \endgroup
  \fi}
\AtBeginDocument{%
\expandafter\renewcommand\csname r@tocindent0\endcsname{0pt}
}
\def\l@subsection{\@tocline{2}{0pt}{2.5pc}{5pc}{}}
\makeatother

\title[Geometric perspective on $\tau$-cluster morphism category]{A geometric perspective on \\ the $\tau$-cluster morphism category}

\author[Schroll]{Sibylle Schroll}
\address{Abteilung Mathematik, Department Mathematik/Informatik der Universit\"at zu K\"oln, Weyertal 86-90, 50931 Cologne, Germany}
\email{schroll@math.uni-koeln.de}

\author[Tattar]{Aran Tattar}
\address{Abteilung Mathematik, Department Mathematik/Informatik der Universit\"at zu K\"oln, Weyertal 86-90, 50931 Cologne, Germany}
\email{atattar@uni-koeln.de}

\author[Treffinger]{Hipolito Treffinger}
\address{IMJ-PRG, Universit\'e Paris Cit\'e, B\^atiment Sophie Germain, 5 rue Thomas Mann, 75205 Paris Cedex 13, France}
\email{treffinger@imj-prg.fr}

\author[Williams]{Nicholas J. Williams}
\address{Department of Mathematics and Statistics, Fylde College, Lancaster University, Lancaster, LA1 4YF, United Kingdom}
\email{nicholas.williams@lancaster.ac.uk}

\begin{document}

\begin{abstract}
We show how the $\tau$-cluster morphism category may be defined in terms of the wall-and-chamber structure of an algebra. This geometric perspective leads to a simplified proof that the category is well-defined.
\end{abstract}

\dedicatory{We dedicate this paper to the memory of pure mathematics at Leicester.}

\thanks{HT and SS were supported by the EPSRC through the Early Career Fellowship, EP/P016294/1. HT was supported by the European Union’s Horizon 2020 research and innovation programme under the Marie Sklodowska-Curie grant agreement No 893654 and by the Deutsche Forschungsgemeinschaft (DFG, German Research Foundation) under Germany's Excellence Strategy Programme -- EXC-2047/1 -- 390685813.
NJW is currently supported by EPSRC grant EP/V050524/1. 
SS is supported by the DFG through the project SFB/TRR 191 Symplectic Structures in
Geometry, Algebra and Dynamics (Projektnummer 281071066-TRR 191).}

\maketitle

\section{Introduction}

The $\tau$-cluster morphism category was introduced under the name `cluster morphism category' by Igusa and Todorov \cite{IT17} for hereditary algebras.
The motivation for the introduction of this category was to give a categorical analogue of the picture space defined in \cite{ITW16}.
Indeed, the classifying space of the $\tau$-cluster morphism category is homeomorphic to the picture space in the hereditary case \cite{IT17}. The introduction of the $\tau$-cluster morphism category allowed Igusa and Todorov to show that the picture space is  $K(\pi, 1)$  for $\pi$ the picture group defined in \cite{ITW16} by showing the classifying space of the $\tau$-cluster morphism category is $K(\pi, 1)$.

Since then, the $\tau$-cluster morphism category has received much attention in the literature. The definition of the category was extended to $\tau$-tilting-finite algebras in \cite{BM21_wide}, where it was given the name `a category of wide subcategories'. The name `$\tau$-cluster morphism category' comes from \cite{HI21}, where some of the results of Igusa and Todorov were generalised. The definition of the category was extended to arbitrary finite-dimensional algebras in \cite{BH21}. The category has also been studied using silting theory in \cite{B21}.

In this paper we show how the $\tau$-cluster morphism category arises naturally in the context of the $g$-vector fan of an algebra. The $g$-vector fan of a finite-dimensional algebra was first studied in \cite{DIJ19}. 
It is defined by taking the two-term presilting complexes  and associating a cone to each, which fit together to form the fan. Cones of two-term presilting complexes nicely encode several properties, such as whether the silting objects contain common summands, as well as reflecting the partial order on them \cite{DIJ19}. The $g$-vector fan is a subfan of the wall-and-chamber structure of an algebra, which arises from stability conditions in the sense of King \cite{King94,BST19,Asai2021}. In the representation-finite hereditary case, the wall-and-chamber structure of the algebra was intersected with a sphere around the origin to give the semi-invariant picture studied in \cite{ITW16}.

\begin{theorem}[Theorem~\ref{thm:cat}, Corollary~\ref{cor:tcm}]
Let $A$ be a finite-dimensional algebra. Then there exists a category $\cat{A}$ defined in terms of the $g$-vector fan of $A$ which is equivalent to the $\tau$-cluster morphism category of $A$.
\end{theorem}

We define the category $\cat{A}$ in Definition~\ref{def:cat} and show in Section 4 that it is equivalent to the $\tau$-cluster morphism category by constructing an intermediate category which is equivalent to both $\cat{A}$ and the $\tau$-cluster morphism category. 
The difficulty in proving that the $\tau$-cluster morphism category is well-defined lies in showing that composition in the category is associative. 
The original proof of this was given in \cite{BM21_wide}. More conceptual proofs of this are in given in \cite{BH21} and \cite{B21}, the latter based on silting theory. In this paper, using the $g$-vector fan, we give a geometrical construction of the $\tau$-cluster morphism category. 
The associativity is then a direct consequence of the construction.
Our definition of the category is motivated by \cite[Proposition~6.5]{MST23}, see Remark~\ref{rmk:stablemorph}.

This paper is structured as follows. We begin in Section~\ref{sec:background} by giving the relevant background of the paper. This consists of background on $\tau$-tilting theory, the $\tau$-cluster morphism category, and the $g$-vector fan of a finite-dimensional algebra. In Section~\ref{sect:category}, we introduce the category defined from the $g$-vector fan of the algebra, which we show to be equivalent to the $\tau$-cluster morphism category in Section~\ref{sect:rel_w_tcm}.

\subsection*{Acknowledgements} 

This paper originated from a working group in  pure mathematics at the University of Leicester.  In solidarity with the other former pure mathematics researchers at the University of Leicester, we dedicate this paper to them and to pure mathematics.

\section{Background}\label{sec:background}

Let $A$ be a finite-dimensional algebra of rank $n$ over a field $K$ and $\mc A$ the category of finitely generated $A$-modules. 
We assume that every subcategory will  be full and closed under isomorphisms. 
A subcategory $\mathcal{X}$ of $\mc A$ is functorially finite if for every object $M \in \mc A$ there are objects $X_M$ and $_MX$ in $\mathcal{X}$  and morphisms $ X_M \to M$ and $M \to {}_MX$ such that for any $Y \in \mathcal{X}$ there are surjections
\[
\Hom_A(Y,X_M) \to \Hom_A(Y, M) \]
\[ \Hom_A( {}_MX,Y) \to \Hom_A(M,Y) 
\]

\subsection{$\tau$-tilting theory}\label{sect:tau_tilt}

In this subsection we give a brief overview of some general results in $\tau$-tilting theory. 
For a more comprehensive survey of $\tau$-tilting theory, see \cite{TSurvey}.

\subsubsection{Torsion pairs}

Torsion pairs were introduced by Dickson to generalise the structure given by torsion and torsion-free abelian groups to arbitrary abelian categories \cite{Dickson1966}. A \emph{torsion pair} is a pair of full subcategories $(\mathcal{T}, \mathcal{F})$ of $\mc A$ such that
\begin{enumerate}
    \item $\Hom_{A}(\mathcal{T}, \mathcal{F}) = 0$;
    \item if $\Hom_{A}(T, \mathcal{F}) = 0$, then $T \in \mathcal{T}$;
    \item if $\Hom_{A}(\mathcal{T}, F) = 0$, then $F \in \mathcal{F}$.
\end{enumerate}
Here $\mathcal{T}$ is called the \emph{torsion class} and $\mathcal{F}$ is called the \emph{torsion-free class}. More generally, a full subcategory $\mathcal{T}$ is called a torsion class if it is a torsion class in some torsion pair, and likewise for torsion-free classes.

\subsubsection{$\tau$-tilting and $\tau$-rigid pairs}

We now define $\tau$-rigid and $\tau$-tilting pairs, following \cite[Definition 0.1 and 0.3]{AIR14}. Let $M$ be an $A$-module and let $P$ be projective in $\mc A$.
We say that $M$ is \emph{$\tau$-rigid} if $\Hom_{A}(M, \tau M)=0$.
The pair $(M,P)$ is said to be \emph{$\tau$-rigid} if $M$ is $\tau$-rigid and $\Hom_{A}(P, M)=0$.
We say moreover that a $\tau$-rigid pair $(M,P)$ is \emph{$\tau$-tilting} if $|M| + |P| = n$.
Here we denote by $|X|$ the number isomorphism classes of direct summands of $X$.        
For two $\tau$-rigid pairs $(M,P)$ and $(N,Q)$ we say that $(M,P)$ \emph{is a direct summand of $(N,Q)$} if $M$ is a direct summand of $N$ and $P$ is a direct summand of $Q$. 

Given a module $M$, we define the two subcategories
\begin{align*}
M^\perp &:= \{X \in \mc A \st \Hom_A(M, X) = 0\},\\
\prescript{\perp}{}{M} &:= \{ X \in \mc A \st \Hom_A(X, M) = 0\}.
\end{align*}
For a $\tau$-rigid pair $(M, P)$, we define two torsion classes $\mathcal{T}_{(M, P)} := \Fac M$ and $\overline{\mathcal{T}}_{(M, P)} := \prescript{\perp}{}{\tau M} \cap P^{\perp}$. 
We have that $\mathcal{T}_{(M, P)} \subseteq \overline{\mathcal{T}}_{(M, P)}$, see \cite[Subsection~2.2]{AIR14}. 
These two torsion classes come in two torsion pairs $(\Fac M, M^{\perp})$ and $(\prescript{\perp}{}{\tau M} \cap P^{\perp}, \Sub \tau M)$. 
We define $\mathcal{F}_{(M, P)} = \Sub \tau M$ and $\overline{\mathcal{F}}_{(M, P)} = M^{\perp}$, where likewise $\mathcal{F}_{(M, P)} \subseteq \overline{\mathcal{F}}_{(M, P)}$. 
We can also construct the so-called \emph{$\tau$-perpendicular subcategory} of $(M, P)$, which was first introduced in \cite{Jasso15}.
This is the category $\tps{}{M, P} := \overline{\mathcal{T}}_{(M, P)} \cap \overline{\mathcal{F}}_{(M, P)} =  M^\perp \cap {}^\perp \tau M \cap P^\perp$, which therefore measures the difference between these two torsion pairs.

A key result in \cite{AIR14} states that there is a bijection between the functorially finite torsion classes and $\tau$-tilting pairs in $\mc A$.
Given a $\tau$-rigid pair $(M,P)$ we say that the $\tau$-tilting pair associated to $\overline{\mathcal{T}}_{(M, P)}$ is the \emph{Bongartz completion} of $(M,P)$. In fact, the Bongartz completion of $(M,P)$ is of the form $(M\oplus T, P)$ for some $\tau$-rigid module $T$. In this case we say that $T$ is the \emph{Bongartz complement} of $(M,P)$.

\subsubsection{$\tau$-tilting reduction}\label{reduction}

It is shown in \cite[Theorem~3.8]{Jasso15} that if $(M, P)$ is a $\tau$-rigid pair, then there is an equivalence of categories 
\begin{equation}\label{phi}\phi\colon \tps{}{M, P} \to \mc B_{(M, P)},\end{equation}
between the $\tau$-perpendicular subcategory and the module category of an algebra $B_{(M, P)}$ that can be constructed explicitly from $(M,P)$.
The process of going from the original algebra $A$ to the algebra $B_{(M, P)}$ is known as \emph{$\tau$-tilting reduction} and the algebra $B_{(M,P)}$ is known as the \emph{$\tau$-tilting reduction algebra} of $A$ by $(M,P)$.

A full subcategory $\mathcal{W}$ of $\mc A$ is said to be \emph{wide} if it is closed under kernels, cokernels and extensions.
An important example of a wide subcategory is the $\tau$-perpendicular subcategory of a $\tau$-rigid pair.
Indeed, it has been shown that $\tps{}{M, P}$ is a functorially finite wide subcategory of $\mc A$ for every $\tau$-rigid pair $(M,P)$ \cite[Corollary 3.22]{BST19} \cite[Theorem~4.12]{DIRRT}. 
Moreover, every wide subcategory is of this form if and only if $A$ is $\tau$-tilting finite, that is, if there are finitely many isomorphism classes of indecomposable $\tau$-rigid modules \cite[Corollary~3.11]{MS17}.

Since the $\tau$-perpendicular subcategories $\tps{}{M, P}$ are equivalent to the module categories $\mc B_{(M,P)}$, they have their own Auslander--Reiten translate $\tau_{\tps{}{M, P}}$.
In this context, given a $\tau_{\tps{}{M, P}}$-rigid pair $(M', P')$ inside $\tps{}{M, P}$, the $\tau_{\tps{}{M, P}}$-perpend\-icular subcategory of $(M', P')$ is denoted $\tps{\tps{}{M, P}}{M', P'}$.

Let $\mathcal{W}=\tps{}{\tilde{M}, \tilde{P}}$ be a functorially finite wide subcategory of $\mc A$, for a $\tau$-rigid pair $(\tilde{M}, \tilde{P})$ in $\mc A$. 
Given a $\tau$-rigid pair $(M, P)$ in $\mathcal{W}$,  let \[\strr{\mathcal{W}}{M, P} := \biggl\{\parbox{3.4cm}{\centering Basic $\tau$-rigid pairs \\ $(N, Q)$ of $\mathcal{W}$} \st \parbox{3.6cm}{\centering $(M, P)$ is a direct \\ summand of $(N, Q)$}\biggr\}.\] 
We further let $\str{\mathcal{W}} := \strr{\mathcal{W}}{0,0}$. Buan and Marsh \cite{BM21_tau,BM21_wide} show how  $\str{\tps{\mathcal{W}}{M, P}}$ is related to $\strr{\mathcal{W}}{M, P}$, as explained in \cite[Section~5]{BH21}. Namely, there is a bijection \[\mathcal{E}_{(M, P)}^{\mathcal{W}}\colon \strr{\mathcal{W}}{M, P} \to \str{\tps{\mathcal{W}}{M, P}}.\]

\subsubsection{The $\tau$-cluster morphism category}\label{subsec:tcmcat}

As we will shortly explain in detail, the $\tau$-cluster morphism category has as its objects the $\tau$-perpendicular subcategories of $\mc A$, with morphisms given by reduction with respect to $\tau$-rigid pairs in these categories. 
Here we follow the approach in \cite{BH21}. Let $A$ be a finite-dimensional algebra. The \emph{$\tau$-cluster morphism category} $\tcm{A}$ is defined as follows.
\begin{enumerate}
\item The objects of $\tcm{A}$ are the $\tau$-perpendicular subcategories of $\mc A$.
\item Given a $\tau$-perpendicular subcategory $\mathcal{W} \subseteq \mc A$ and a basic $\tau$-rigid pair $(M, P)$ in $\mathcal{W}$, we define a formal symbol $g_{(M, P)}^{\mathcal{W}}$.
\item For two $\tau$-perpendicular subcategories $\mathcal{W}_{1}$ and $\mathcal{W}_{2}$ of $\mc A$, define 
\[\Hom_{\tcm{A}}(\mathcal{W}_{1}, \mathcal{W}_{2}) = \biggl\{ g_{(M, P)}^{\mathcal{W}_{1}} \st \parbox{6cm}{\centering $(M, P)$ is a basic $\tau$-rigid pair in $\mathcal{W}_{1}$ and $\mathcal{W}_{2} = \tps{\mathcal{W}_{1}}{M, P}$} \biggr\}.\]
\item Given $g_{(M, P)}^{\mathcal{W}_{1}}\colon \mathcal{W}_{1} \to \mathcal{W}_{2}$ and $g_{(N, Q)}^{\mathcal{W}_{2}}\colon \mathcal{W}_{2} \to \mathcal{W}_{3}$ in $\tcm{A}$, we denote \[(\widetilde{N}, \widetilde{Q}) := \left( \mathcal{E}_{(M, P)}^{\mathcal{W}_{1}} \right)^{-1}(N, Q).\] The composition of the two morphisms is then defined as \[g_{(N, Q)}^{\mathcal{W}_{2}} \circ g_{(M, P)}^{\mathcal{W}_{1}} = g_{(M \oplus \widetilde{N}, P \oplus \widetilde{Q})}^{\mathcal{W}_{1}}.\]
\end{enumerate}

\subsection{The wall-and-chamber structure of an algebra}\label{sect:back:w&c}

The $\tau$-tilting theory of a finite-dimensional algebra with $n$ isomorphism classes of simple modules $\{S(1), \dots, S(n)\}$ is related to a certain wall-and-chamber structure of $\mathbb{R}^n$, as we now explain. We will interpret the $\tau$-cluster morphism category in terms of this structure.

We denote by $K_{0}(A)$ the Grothendieck group of $\mc A$. This is a free abelian group of rank $n$.
Given an $A$-module $M$, we write $[M]$ for the class of $M$ in $K_{0}(A)$, which we identify with a vector in $\mathbb{Z}^n$ via the isomophism $\Phi: K_0(A) \to \mathbb{Z}^n$ defined by $\Phi([S(i)]) = e_i$ where  $\{e_1, \dots, e_n\}$ is the canonical basis of $\mathbb{R}^n$.
If $A=KQ/I$ is a bounded path algebra of a quiver $Q$, we have  $[M] = \dimu M$, the dimension vector of $M$ as a quiver representation. 
In this paper we write $\dimu M = \Phi([M])$.
By $\ip{-}{-}$, we mean the standard inner product on $\mathbb{R}^{n}$.

Recall the notion of stability from \cite{King94}. Given $v \in \mathbb{R}^{n}$, we say that a non-zero $A$-module $M$ is \emph{$v$-semistable} if $\ip{v}{\dimu M} = 0$ and $\ip{v}{\dimu N} \geqslant 0$ for every factor module $N$ of $M$. 
If $M$ is $v$-semistable and $\ip{v}{\dimu N} \neq 0$ for all proper factor modules $N$ of $M$
, we say that $M$ is \emph{$v$-stable}. The \emph{stability space} of an $A$-module $M$ is then defined to be \[\D(M) := \{v \in \mathbb{R}^{n} \st M \text{ is }v\text{-semistable}\}.\] The \emph{wall-and-chamber} structure of the algebra $A$ is the cone complex \[\bigcup_{M
 \in \mc A \setminus \{0\}}\D(M).\] Intersecting this cone complex with a sphere around the origin gives what was called the ``semi-invariant picture'' in the representation-finite hereditary case in \cite{ITW16}.

To investigate the wall-and-chamber structure, it is useful to consider the following torsion and torsion-free classes from \cite[Subsection 3.1]{BKT}---see also \cite[Lemma~6.6]{BridgelandScat}.
For $v\in \mathbb{R}^n$, we have the torsion classes 
\[\overline{\T}_{v}=\{M\in \mc A \st \langle v,\dimu N\rangle \geq 0 \text{ for every quotient $N$ of $M$}\}\] 
and 
\[\T_{v}=\{M\in \mc A \st \langle v, \dimu N\rangle > 0 \text{ for every quotient $N \neq 0$ of $M$}\},\] 
and we have the torsion-free classes 
\[\overline{\F}_{v}=\{M\in \mc A \st \langle v, \dimu L\rangle \leqslant 0 \text{ for every submodule $L$ of $M$}\}\] 
and 
\[\F_{v} = \{M\in \mc A \st \langle v,\dimu L\rangle < 0 \text{ for every submodule $L \neq 0$ of $M$}\}.\]
Moreover, both $(\overline{\T}_{v}, \F_{v})$ and $(\T_{v}, \overline{\F}_{v})$ are torsion pairs \cite[Proposition~3.1]{BKT}.
Following \cite{Asai2021}, we say that $v, v' \in \mathbb{R}^n$ are \emph{TF-equivalent} if $\overline{\T}_v = \overline{\T}_{v'}$ and $\overline{\F}_v = \overline{\F}_{v'}$.
It is clear that TF-equivalence is an equivalence relation.
Moreover, it was shown in \cite[Lemma~2.14]{Asai2021} that every TF-equivalence class is convex, and hence connected, in $\mathbb{R}^n$.
The category of $v$-semistable objects is $\mathcal{W}_{v} = \overline{\mathcal{T}}_{v} \cap \overline{\mathcal{F}}_{v}$.
It follows from \cite[Proposition~3.24]{BST19} that $\W_v$ is always a wide subcategory of $\mc A$.
Note that, by definition $\overline{\T}_v = \overline{\T}_{v'}$ and $\overline{\F}_v = \overline{\F}_{v'}$ for every $v,v'$ in every TF-equivalence class $E$. 
By abuse of notation, we denote by $\overline{\T}_E$ the torsion class $\overline{\T}_v$ for any $v \in E$. 
Likewise, we denote by $\overline{\F}_E$ the torsion-free class $\overline{\F}_v$ for every $v \in E$.
In particular, we can associate to each TF-equivalence $E$ the subcategory $\W_E = \overline{\T}_E\cap  \overline{\F}_E \subset \mc A$.
These subcategories will be instrumental in defining the $\tau$-cluster morphism category from the wall-and-chamber structure.

\subsubsection{From $\tau$-tilting theory to the wall-and-chamber structure}\label{sect:tau_wcs}

Let $M$ be an $A$-module. 
Choose the minimal projective presentation \[P_{-1}\longrightarrow P_0\longrightarrow M\longrightarrow 0\] of $M$, where $P_0=\bigoplus_{i=1}^n P(i)^{a_i}$ and $P_{-1} = \bigoplus_{i=1}^n P(i)^{b_i}$ and $\{P(1), P(2),$ $\dots, P(n)\}$ is a complete set of isomorphism-class representatives of the indecomposable projective $A$-modules. 
Then the \emph{$g$-vector} of $M$ is defined as \[g^M=(a_1-b_1, a_2-b_2,\dots, a_n-b_n).\]
The $g$-vector of a $\tau$-rigid pair $(M,P)$ is defined as $ g^{M}-g^{P}$.

\begin{remark}
We note that $g$-vectors can also viewed as the elements of the Grothendieck group of an extriangulated category $K^{[-1,0]}(\operatorname{proj} A)$ which is naturally associated to $A$, see \cite[Proof of Proposition~4.44]{PPPP}. 
\end{remark}

Consider now a basic $\tau$-rigid pair $(M,P)$ where $M=\bigoplus_{i=1}^kM_i$ and $P=\bigoplus_{j=k+1}^t P_j$ are the decomposition of $M$ and $P$ as sums of indecomposable modules, respectively. 
We define the polyhedral cones $\mathcal{C}_{(M,P)}$ and $\overline{\mathcal{C}}_{(M,P)}$  associated to $(M,P)$ to be the sets \[\mathcal{C}_{(M,P)}=\biggl\{\sum_{i=1}^{k} \alpha_i g^{M_i}-\sum_{j=k+1}^{t} \alpha_j g^{P_j} \st \alpha_i > 0 \text{ for every $1\leqslant i\leqslant t$}\biggr\},\] 
\[\overline{\mathcal{C}}_{(M,P)}=\biggl\{\sum_{i=1}^{k} \alpha_i g^{M_i}-\sum_{j=k+1}^{t} \alpha_j g^{P_j} \st \alpha_i \geq 0 \text{ for every $1\leqslant i\leqslant t$}\biggr\},\]
where $\{g^{M_1},\dots, g^{M_k}, -g^{P_{k+1}}, \dots, -g^{P_t}\}$ is the set of $g$-vectors for the indecomposable summands of $(M,P)$. 
Note that $\overline{\mathcal{C}}_{(M,P)}$ coincides with the closure of $\mathcal{C}_{(M,P)}$ with respect to the canonical topology in $\mathbb{R}^n$.
It is shown in \cite{DIJ19} that the set
\[  \bigcup_{(M,P) \in \str{A}} \overline{\mathcal{C}}_{(M,P)} \]
forms a polyhedral fan in $\mathbb{R}^n$.

It is shown in \cite{BST19, Asai2021} that if $(M, P)$ is a $\tau$-rigid pair, then the cone $\mathcal{C}_{(M,P)}$ is a TF-equivalence class and, moreover, \[\W_{\mathcal{C}_{(M,P)}}= \tps{}{M, P}.\] That is, the wide subcategory associated to the cone $\mathcal{C}_{(M, P)}$ is the $\tau$-perpendicular subcategory of $(M, P)$. Furthermore, \cite[Theorem~4.7]{Asai2021} shows that an algebra is $\tau$-tilting-finite if and only if every TF-equivalence class is of the form $\mathcal{C}_{(M, P)}$ for a $\tau$-rigid pair $(M, P)$.

\subsubsection{$\tau$-tilting reduction and the wall-and-chamber structure}\label{sect:red_wcs}

The relation between the wall-and-chamber structures and $\tau$-tilting reduction is studied in \cite[Section~4]{Asai2021}, as we now explain. See also \cite{AHIKM2022}.
Following \cite[Section~4]{Asai2021}, for a $\tau$-rigid pair $(M, P)$, we define a subset $N_{(M, P)} \subset \mathbb{R}^{n}$ by \[N_{(M, P)} := \{v \in \mathbb{R}^{n} \st \mathcal{T}_{(M, P)} \subseteq \mathcal{T}_{v} \subseteq \overline{\mathcal{T}}_{v} \subseteq \overline{\mathcal{T}}_{(M, P)}\}.\] If $v \in N_{(M, P)}$, then $\overline{\mathcal{F}}_{v} \subseteq \overline{\mathcal{F}}_{(M, P)}$, and so $\mathcal{W}_{v} \subseteq \tps{}{M, P}$. It is clear from the definition that $N_{(M, P)}$ is a union of TF-equivalence classes in~$\mathbb{R}^{n}$. 
It can be thought of as the union of the TF-equivalence classes surrounding $\mathcal{C}_{(M, P)}$.

Let $B = B_{(M, P)}$ be the $\tau$-tilting reduction of $A$ with respect to $(M, P)$. Further, let $\{X_{1}, X_{2}, \dots, X_{m}\}$ be the simple objects of $\tps{}{M, P}$. 
When we use the term `simple object', we mean the simple objects of $\tps{}{M, P}$ as an abelian category, rather than the simple $A$-modules which lie in $\tps{}{M, P}$.
There is a linear map $\pi = \pi_{(M, P)}\colon \mathbb{R}^{n} \to \mathbb{R}^{m}$ defined 
\begin{equation}\label{eq:pi}
\pi(v)_{i} = \frac{\ip{v}{\dimu X_{i}}}{d_{i}},
\end{equation}
where $\pi(v)_{i}$ means the $i$-th coordinate of $\pi(v)$ and $d_{i} = \dim_{K} \End_{A}(X_{i})$. The map $\pi$ has the following properties \cite[Lemma~4.4, Theorem~4.5]{Asai2021}, recalling from Subsection~\ref{reduction} \eqref{phi} the equivalence of categories $\phi$:
\begin{enumerate}
\item The restriction $\pi|_{N_{(M, P)}}\colon N_{(M, P)} \to \mathbb{R}^{m}$ is surjective.
\item For any $v \in N_{(M, P)}$, we have
\begin{align*}
\phi(\overline{\mathcal{T}}_{v}) &= \overline{\mathcal{T}}_{\pi(v)}, & \phi(\mathcal{F}_{v}) &= \mathcal{F}_{\pi(v)}, \\
\phi(\mathcal{T}_{v}) &= \mathcal{T}_{\pi(v)}, & \phi(\overline{\mathcal{F}}_{v}) &= \overline{\mathcal{F}}_{\pi(v)}, & \phi(\mathcal{W}_{v}) &= \mathcal{W}_{\pi(v)}.
\end{align*}
\item For any $v \in N_{(M, P)}$ and $L \in \tps{}{M, P}$, the wall $\D(\phi(L))$ coincides with $\pi(\D(L) \cap N_{(M, P)})$.
\item The map $\pi$ induces a bijection between TF-equivalence classes in $N_{(M, P)}$ and TF-equivalence classes for $\mc B_{(M, P)}$ in $\mathbb{R}^{m}$.
\end{enumerate}
This interpretation of $\tau$-tilting reduction will be key to our construction of the $\tau$-cluster morphism category in terms of the wall-and-chamber structure.

\section{A category associated to the wall-and-chamber structure}\label{sect:category}

We begin by constructing a poset from the set of TF-equivalence classes of the form $\mathcal{C}_{(M,P)}$ in the wall-and-chamber structure for a $\tau$-rigid pair $(M,P)$.
We then use this poset to construct a category $\cat{A}$, which we later show to be equivalent to the $\tau$-cluster morphism category.
To this end, we denote by $TF_A$ the set of all TF-equivalence classes in the wall-and-chamber structure of $A$ of the form $\mathcal{C}_{(M,P)}$ for a $\tau$-rigid pair $(M,P)$ in $\mc A$.

\begin{proposition}
The relation $E \leqslant E'$ if $E \subseteq \overline{E'}$ for TF-equivalence classes $E, E' \in TF_A$ induces a partial order on $TF_A$.
\end{proposition}

\begin{proof}
It is clear that the relation $\leqslant$ is reflexive.
To show that the relation $\leqslant$ is transitive, suppose that $E, E', E'' \in TF_A$ such that $E\leqslant E'$ and $E' \leqslant E''$.
Then $E \subset \overline{E'} \subset \overline{\overline{E''}}= \overline{E''}$, and so $E \leqslant E''$. To show anti-symmetry, note that, since the TF-equivalence classes are disjoint, we have that if $E \leqslant E'$, then $E \subseteq \overline{E'} \setminus E'$, and so $E$ has dimension strictly smaller than $E'$. This implies that the relation $\leqslant$ must be anti-symmetric.
\end{proof}
Note that this is in fact the standard partial order on the strata of a stratified topological space---see, for instance, \cite[Section~2.1]{Woolf10}.

It is a well-known fact that every poset can be seen as a category where the objects of the category correspond to the elements of the set. The morphisms are determined by the partial order: that is, there is a unique morphism $E \to E'$ whenever $E \leqslant E'$. 
In particular, we have that $TF_A$ with the partial order defined above gives rise to a category.
Note that in this case the category $TF_A$ always has an initial object, namely the TF-equivalence $\mathcal{C}_{(0,0)}$, consisting only of the origin of $\mathbb{R}^n$, and no terminal object.
We write $f_{EE'}$ for the unique morphism from $E$ to $E'$ which exists when $E \leqslant E'$.

\begin{lemma}\label{lem:order=N}
Let $E,E' \in TF_A$. 
Then $E \leqslant E'$ if and only if $E' \subseteq N_{E}$.
\end{lemma}
\begin{proof}
Let $E$ and $E'$ be TF-equivalence classes in $TF_{A}$ such that $E \leqslant E'$. 
By definition of $TF_{A}$, $E = \mathcal{C}_{(M, P)}$ and $E' = \mathcal{C}_{(M', P')}$ for some $\tau$-rigid pairs $(M, P)$ and $(M', P')$. We have that $\mathcal{C}_{(M, P)} \subseteq \overline{\mathcal{C}}_{(M', P')}$. Hence, by taking limits inside $E'$, we have that $\overline{\mathcal{T}}_{E'} \subseteq \overline{\mathcal{T}}_{E}$ and $\overline{\mathcal{F}}_{E'} \subseteq \overline{\mathcal{F}}_{E}$. 
Indeed, given $M \in \overline{\mathcal{T}}_{E'}$, we have that $\ip{v}{\dimu N} \geqslant 0$ for every quotient $N$ of $M$ and all $v \in E'$. Since any $w \in E$ is a limit of a sequence $\{v_{k}\}_{k \in \mathbb{N}} \subset E'$, we must have that $\ip{w}{\dimu N} \geqslant 0$ for every quotient $N$ of $M$ and all $w \in E$ as well. The argument for torsion-free classes is similar.
 The inclusion of torsion-free classes here implies that $\mathcal{T}_{E} \subseteq \mathcal{T}_{E'}$, and so we obtain that \[\mathcal{T}_{E} \subseteq \mathcal{T}_{E'} \subseteq \overline{\mathcal{T}}_{E'} \subseteq \overline{\mathcal{T}}_{E},\] which precisely gives us that $E' \subseteq N_{E}$.

To show the converse, suppose that $E' \subseteq N_{E}$.
Then, by definition, we have that 
\[\mathcal{T}_{E} \subseteq \mathcal{T}_{E'} \subseteq \overline{\mathcal{T}}_{E'} \subseteq \overline{\mathcal{T}}_{E}.\]
Moreover, there are $\tau$-rigid pairs $(M,P)$ and $(M',P')$ such that $E = \mathcal{C}_{(M,P)}$ and $E' = \mathcal{C}_{(M',P')}$.
It follows from \cite[Proposition~2.9]{AIR14} that $(M,P)$ is a direct summand of the $\tau$-tilting pairs $(T, Q)$ and $(\overline{T}, \overline{Q})$ corresponding to $\mathcal{T}_{E'}$ and $\overline{\mathcal{T}}_{E'} $, respectively. 
But it also follows from \cite[Proposition~2.9]{AIR14} that the maximal common direct summand of $(T, Q)$ and $(\overline{T}, \overline{Q})$ is precisely $(M',P')$.
Hence $(M,P)$ is a direct summand of $(M',P')$.
Then by construction we obtain that $\mathcal{C}_{(M,P)} \subset \overline{\mathcal{C}}_{(M',P')}$.
In other words, $E \leqslant E'$.
\end{proof}

Given a TF-equivalence class $E$, we write $\np{E}\colon \mathbb{R}^{n} \to \lspan{E}^{\perp}$ for the projection onto the orthogonal complement of the vector subspace $\lspan{E}$. We now define our category $\cat{A}$. 

\begin{definition}\label{def:cat}
We define the category $\cat{A}$ as follows.
\begin{enumerate}[label=(\Alph*),wide]
\item The objects of $\cat{A}$ are equivalence classes $[E]$ of objects of $TF_{A}$ under the equivalence relation where $E \sim E'$ if $\W_E = \W_{E'}$, recalling that these are the wide subcategories associated to the TF-equivalence classes in Subsection~\ref{sect:back:w&c}.
\item Given objects $[E]$ and $[F]$ of $\cat{A}$, we have that $\Hom_{\cat{A}}([E], [F])$ consists of equivalence classes of objects in \[\bigcup_{E' \in [E], F' \in [F]} \Hom_{TF_{A}}(E', F')\] under the equivalence relation where $f_{EF} \sim f_{E'F'}$ if and only if $\np{E}(F) = \np{E'}(F')$. 
Recall that the $\Hom$-set $\Hom_{TF_{A}}(E', F')$ equals $\{f_{E'F'}\}$ if $E' \leqslant F'$, and is empty otherwise.\label{op:cat:morph_equiv}
\item Given a morphism $[f_{EF}] \in \Hom_{\cat{A}}([E], [F])$ and a morphism $[f_{FG}] \in \Hom_{\cat{A}}([F], [G])$, the composition $[f_{FG}] \circ [f_{EF}]$ is defined to be $[f_{EG}]$.
\end{enumerate}
\end{definition}

\begin{remark}
The equivalence relations on objects and morphisms of $TF_{A}$ to form the category $\cat{A}$ coincide with the gluing rules used to construct the picture space \cite[Definition~3.2.1]{ITW16}.
\end{remark}

\begin{remark}\label{rmk:stablemorph}
Morphisms in the $\tau$-cluster morphism category are given by the so-called signed $\tau$-exceptional sequences introduced in \cite{BM21_tau}, see also \cite{MT20}. 
The construction of $\cat{A}$ in Definition~\ref{def:cat} is motivated by \cite[Proposition~6.5]{MST23} where it was shown, in the notation of Subsection~\ref{subsec:tcmcat}, that if $\mathcal{W}_1 = \mathcal{J}(M',P')$ and $g_{(M, P)}^{\mathcal{W}_{1}}\colon \mathcal{W}_{1} \to \mathcal{W}_{2}$ is a morphism in $\tcm{A}$, then $M$ and $P$ are $v$-semistable objects for every $v \in \mathcal{C}_{(M',P')}$.
\end{remark}

Note that it is not yet clear that composition is well-defined, for two reasons.
\begin{enumerate}
\item It is not clear how to compose morphisms $[f_{EF}]$ and $[f_{F'G}]$ where $F \sim F'$. In order to be able to do this, one would need to find $TF$-equivalence classes $E' \in [E]$, $F'' \in [F]$, $G'\in [G]$ and morphisms $f_{E'F''} \sim f_{EF}$ and $f_{F''G'} \sim f_{F'G}$, which would give the composition as $[f_{E'G'}]$.\label{prob:1}
\item It is not clear that composition respects the equivalence relation. For instance, given $f_{EF} \sim f_{E'F'}$ and $f_{FG} \sim f_{F'G'}$, it is not clear that $f_{EG} \sim f_{E'G'}$.\label{prob:2}
\end{enumerate}

In order to resolve these issues, we first show that equivalent TF-equivalence classes have the same linear span. This means that the projection maps onto their orthogonal complements are also the same. Hence, it makes sense to compare $\np{E}(F)$ and $\np{E'}(F')$ when $E \sim E'$. In order to show this, we show how the linear span of a TF-equivalence class may be described in terms of the associated wide subcategory.

\begin{lemma}\label{lem:tf_orth}
Let $E$ be a TF-equivalence class. Then \[\{\dimu X \st X \text{ a simple object in } \mathcal{W}_{E}\}\] is a basis of $\lspan{E}^{\perp}$.
\end{lemma}
\begin{proof}
We use the fact that $E = \mathcal{C}_{(M, P)}$ for a $\tau$-rigid pair $(M, P)$. 
We then have that $\lspan{E}$ is the span of the $g$-vectors of the indecomposable summands of $(M, P)$. These $g$-vectors are linearly independent by \cite[Theorem~5.1]{AIR14}. Hence $\dim \lspan{E} = |M| + |P|$, and so $\dim \lspan{E}^{\perp} = n - |M| - |P|$.

We then note that $\tps{}{M, P}$ is equivalent to $\mc B_{(M, P)}$, the category of modules over the $\tau$-tilting reduction algebra. 
This moreover induces an isomorphism of Grothendieck groups $K_{0}(\tps{}{M, P}) \cong K_{0}(\mc B_{(M, P)})$. 
We then have that $K_{0}(\mc B_{(M, P)}) \cong \mathbb{Z}^{n - |M| - |P|}$ with a basis given by the dimension vectors of the simple modules, and so $K_{0}(\tps{}{M, P}) = K_{0}(\mathcal{W}_{E}) \cong \mathbb{Z}^{n - |M| - |P|}$ with a basis given by the dimension vectors of the simple objects.
The result then follows from the fact that $K_{0}(\tps{}{M, P})  \subseteq \lspan{E}^{\perp}$, by definition of $\mathcal{W}_{E}$.
\end{proof}

\begin{corollary}\label{cor:tf_linear_eq}
Let $E$ and $E'$ be TF-equivalence classes such that $[E] = [E']$. Then
\begin{enumerate}
\item $\lspan{E}^{\perp} = \lspan{\dimu M \st M \in \mathcal{W}_{E}}$;\label{op:tf_linear_eq:orth_desc}
\item $\lspan{E}^{\perp} = \lspan{E'}^{\perp}$;\label{op:tf_linear_eq:orth}
\item $\lspan{E} = \lspan{E'}$;\label{op:tf_linear_eq:span}
\item $\np{E} = \np{E'}$.\label{op:tf_linear_eq:map}
\end{enumerate}
\end{corollary}
\begin{proof}
Claim~\eqref{op:tf_linear_eq:orth_desc} follows from Lemma~\ref{lem:tf_orth}. Indeed, it is obvious that \[\lspan{\dimu X \st X \text{ a simple object in } \mathcal{W}_{E}} \subseteq \lspan{\dimu M \st M \in \mathcal{W}_{E}},\] 
whilst the definition of $\mathcal{W}_{E}$ gives us that \[\lspan{E}^{\perp} \supseteq \lspan{\dimu M \st M \in \mathcal{W}_{E}}.\]
Statement \eqref{op:tf_linear_eq:orth} then follows from \eqref{op:tf_linear_eq:orth_desc}, since if $[E] = [E']$, then $\mathcal{W}_{E} = \mathcal{W}_{E'}$. Statements~\eqref{op:tf_linear_eq:span} and \eqref{op:tf_linear_eq:map} are then easy consequences.
\end{proof}

We show that using the orthogonal projection $\nu$ is equivalent to using the map $\pi$ from Subsection~\ref{sect:red_wcs}.

\begin{lemma}\label{lem:proj=red}
Let $E$ and $E'$ be TF-equivalence classes such that $E \sim E'$ with $E \leqslant F$ and $E' \leqslant F'$ for some TF-equivalence classes $F$ and $F'$. Then $\np{E}(F) = \np{E'}(F')$ if and only if $\pi_{E}(F) = \pi_{E'}(F')$.
\end{lemma}

\begin{proof}
First let $\{X_{1}, X_{2}, \dots, X_{m}\}$ be the set of simple objects of $\mathcal{W}_{E} = \mathcal{W}_{E'}$ with $d_{i} = \dim \End_{A}X_{i}$. Then let $(M, P)$ be the $\tau$-rigid pair with $E = \mathcal{C}_{(M, P)}$. Furthermore let $T = T_{1} \oplus \dots \oplus T_{m}$ be the Bongartz complement of $(M, P)$.
We denote the $g$-vectors of $T_{1}, T_{2}, \dots, T_{m}$ by $g_{1}, g_{2}, \dots, g_{m}$, and the $g$-vectors of the indecomposable direct summands of $(M, P)$ by $g_{m + 1}, g_{m + 2}, \dots, g_{n}$. By \cite[Theorem~5.1]{AIR14}, $\{g_{1}, g_{2}, \dots, g_{n}\}$ forms a basis of $\mathbb{R}^n$.

We will describe $\np{E}$ using this basis, and then use this to compare $\np{E}$ to $\pi_{E}$. Note first that $\ip{g_{i}}{\dimu X_{j}} = 0$ for any $m + 1 \leqslant i \leqslant n$, since $g_{i} \in \lspan{E}$ and $\dimu X_{j} \in \lspan{E}^{\perp}$. 
Moreover, $\ip{g_{i}}{\dimu X_{j}} = d_{j}\delta_{ij}$ for $1 \leqslant i \leqslant m$ by, for instance, \cite[Proof of Lemma~4.4(2)]{Asai2021}, see also \cite[Lemma~3.3]{Treffinger2019}.
Hence, we have that \[\nu(g_{i}) = \sum_{j = 1}^m \frac{\ip{g_{i}}{\dimu X_{j}}}{d_{j}}\np{E}(g_{j})\] for all $i$. 
This implies that \[\nu(v) = \sum_{j = 1}^m \frac{\ip{v}{\dimu X_{j}}}{d_{j}}\np{E}(v)\] for all $v \in \mathbb{R}^n$, as $\{g_{1}, g_{2}, \dots, g_{m}\}$ is a basis. 
Moreover, since $\np{E}(g_{i}) = 0$ for $m + 1 \leqslant i \leqslant n$, we have that $\lspan{E}^\perp$ must have basis $\{\np{E}(g_{1}), \np{E}(g_{2}), \dots, \np{E}(g_{m})\}$, as the image of $\np{E}$ must be the whole of $\lspan{E}^\perp$, which has dimension $m$. 
Hence, let $\rho_{E} \colon \lspan{E}^\perp \to \mathbb{R}^m$ be the isomorphism of vector spaces sending $\np{E}(g_{i}) \mapsto e_{i}$.

Note that $\{\np{E}(g_1), \dots, \np{E}(g_m)\}$ is the unique basis of $\lspan{E}^\perp$ such that $\ip{\np{E}(g_i)}{\dimu X_j}= d_j\delta_{ij}$. 
Then this basis depends only on $\mathcal{W}_E = \mathcal{W}_{E'}$. 
It is then clear from the definition of $\pi_{E}$ from Subsection~\ref{sect:red_wcs} that $\pi_{E} = \rho_{E}\np{E}$. 
Then because $\rho_{E}$ only depends upon $\mathcal{W}_{E} = \mathcal{W}_{E'}$ and $\lspan{E} = \lspan{E'}$, we also have that $\pi_{E'} = \rho_{E}\np{E'}$. Since $\rho_{E}$ is an isomorphism, it follows that $\np{E}(F) = \np{E'}(F')$ if and only if $\pi_{E}(F) = \pi_{E'}(F')$.
\end{proof}

We now show that our category $\cat{A}$ is in fact a well-defined category. We first solve problem~\eqref{prob:1}.

\begin{lemma}
Given morphisms $[f_{EF}]$ and $[f_{F'G'}]$ where $F \sim F'$, there exists a morphism $f_{FG} \sim f_{F'G'}$ with $G \sim G'$.
\end{lemma}
\begin{proof}
Since $F \sim F'$, we know that the projection of the fan $N_{F}$ under $\np{F}$ must be equal to the projection of the fan $N_{F'}$ under $\np{F'}$ by Lemma~\ref{lem:proj=red} and the properties of $\pi$ described in Subsection~\ref{sect:red_wcs}.
 Hence, we must have that $\np{F'}(G')$ must be equal to $\np{F}(G)$ for some cone $G$ in $N_{F}$.  Since then $F \leqslant G$ by Lemma~\ref{lem:order=N}, this then gives the morphism $f_{FG}$ such that $f_{FG} \sim f_{F'G'}$.
\end{proof}

Now we solve problem~\eqref{prob:2}.

\begin{lemma}
Let $[E]$, $[F]$, and $[G]$ be objects of $\cat{A}$ with morphisms $[f_{EF}]$ and $[f_{FG}]$. Suppose that we further have $E' \in [E]$, $F' \in [F]$, and $G' \in [G]$, and that there are morphisms $f_{E'F'} \in [f_{EF}]$ and $f_{F'G'} \in [f_{FG}]$. Then $[f_{EG}] = [f_{E'G'}]$.
\end{lemma}
\begin{proof}
We must show that $f_{E'G'} \sim f_{EG}$, that is, $\np{E'}(G') = \np{E}(G)$. Since $\np{E}(F) = \np{E'}(F')$, we may choose $w \in F$ and $w' \in F'$ such that $\np{E}(w) = \np{E'}(w')$. Then, let $v \in \np{F}(G) = \np{F'}(G')$. 

The generating vectors of $G$ consist of those of $F$ along with other vectors which have components in $\lspan{F}$ and its orthogonal complement. 
Hence, since $v \in \np{F}(G)$ and $w \in F$, there exists $\epsilon > 0$ such that $w + \epsilon v \in G$. 
Indeed, the vectors in $\np{F}(G)$ are those which are orthogonal to $F$ and point into $G$ from any point in $F$, recalling that all these cones are open.
Likewise, there exists $\epsilon' > 0$ such that $w' + \epsilon'v \in G'$. If we take $\delta = \min\{\epsilon, \epsilon'\}$, then we have both $w + \delta v \in G$ and $w' + \delta v \in G'$. We then obtain that
\begin{align*}
\np{E}(w + \delta v) &= \np{E}(w) + \delta\np{E}(v) \\
&= \np{E'}(w') + \delta\np{E'}(v) \\
&= \np{E'}(w' + \delta v).
\end{align*}
Thus $\np{E}(G) \cap \np{E'}(G') \neq \varnothing$. 
The images of cones under $\np{E}$ and $\np{E'}$ are either disjoint or equal by Lemma~\ref{lem:proj=red} and Subsection~\ref{sect:red_wcs}.
Hence, we conclude that we must have $\np{E}(G) = \np{E'}(G')$. This implies that $f_{E'G'} \sim f_{EG}$, as desired.
\end{proof}

As a consequence we have the following.

\begin{theorem}\label{thm:cat}
The set of equivalence classes $[E]$ of objects of $TF_{A}$ together with the morphisms defined as in Definition~3.3 gives rise to a well-defined category~$\cat{A}$.
\end{theorem}

\begin{example} \label{example1}
Let $Q$ be the quiver \[\begin{tikzcd}
1 \arrow[r, "\alpha", bend left] & 2 \arrow[l, "\beta", bend left]
\end{tikzcd}\] and let $A = KQ/ \langle \beta \alpha \rangle$.
The Auslander--Reiten quiver of $A$ can be found in Figure~\ref{fig:ARquiver}, its wall-and-chamber structure in Figure~\ref{fig:w&c} and its $g$-vector fan in Figure~\ref{fig:g-fan}.
\begin{figure}
\[
\begin{tikzpicture}[scale=1.7]

\node(2l) at (0,0) {\fbox{$\rep{2}$}};
\node(1) at (2,0) {$\rep{1}$};
\node(2r) at (4,0) {\fbox{$\rep{2}$}};
\node(12) at (1,1) {$\rep{1\\2}$};
\node(21) at (3,1) {$\rep{2\\1}$};
\node(212) at (2,2) {$\rep{2\\1\\2}$};

\draw[->] (2l) -- (12);
\draw[->] (12) -- (212);
\draw[->] (12) -- (1);
\draw[->] (1) -- (21);
\draw[->] (212) -- (21);
\draw[->] (21) -- (2r);

\draw[->,dashed] (1) -- (2l);
\draw[->,dashed] (2r) -- (1);
\draw[->,dashed] (21) -- (12);

\end{tikzpicture}
\]
\caption{The Auslander--Reiten quiver of $A$.}
\label{fig:ARquiver}
\end{figure}
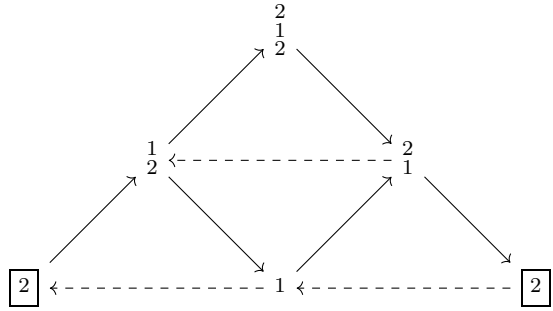

\begin{figure}
\centering
\begin{tikzpicture}
\draw[thick] (0,3) -- (0,-3);
\draw[thick] (-3,0) -- (3,0);
\draw[thick] (-3,3) -- (3,-3);
\draw (0,2.8) node[anchor=west]{$\mathcal{D}(\rep{1})$};
\draw (0,-2.8) node[anchor=east]{$\mathcal{D}(\rep{1})$};
\draw (2.6, 0) node[anchor=south]{$\mathcal{D}(\rep{2})$};
\draw (-2.6, 0) node[anchor=north]{$\mathcal{D}(\rep{2})$};
\draw (2.6, -2.4) node[anchor=south]{$\mathcal{D}\bigl(\rep{1\\2}\bigr)$};
\draw (-2.4, 2.6) node[anchor=north east]{$\mathcal{D}\bigl(\rep{2\\1}\bigr)$};

\end{tikzpicture}
\caption{The wall-and-chamber structure of $A$.}
\label{fig:w&c}
\end{figure}
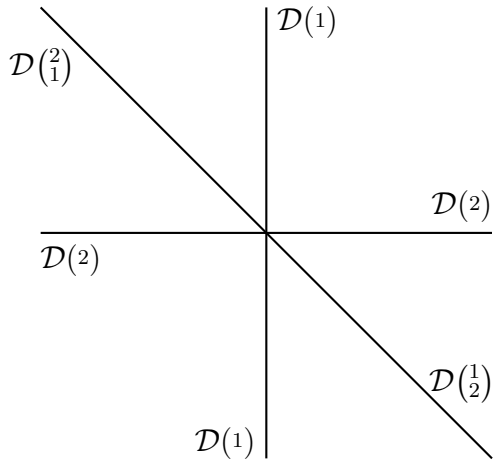

\begin{figure}
\centering
\begin{tikzpicture}
\draw[thick, ->, blue] (0,0) -- (0,-3);
\draw[thick, ->, blue] (0,0) -- (0,3);
\draw[thick, ->, blue] (0,0) -- (-3,0);
\draw[thick, ->, blue] (0,0) -- (3, 0);
\draw[thick, ->, blue] (0,0) -- (3,-3);
\draw[thick, ->, blue] (0,0) -- (-3,3);
\filldraw[blue] (0,0) circle (2pt);
\draw (0,2.7) node[anchor=west]{$g^{\rep{2\\1\\2}}$};
\draw (0,-2.5) node[anchor=east]{$-g^{\rep{2\\1\\2}}$};
\draw (2.6, 0) node[anchor=south]{$g^{\rep{1\\2}}$};
\draw (-2.6, 0) node[anchor=north]{$-g^{\rep{1\\2}}$};
\draw (2.6, -2.4) node[anchor=south]{$g^{\rep{1}}$};
\draw (0, 0) node[anchor=south west]{$g^{\rep{0}}$};
\draw (-2.4, 2.6) node[anchor=north east]{$g^{\rep{2}}$};

\end{tikzpicture}
\caption{The $g$-vector fan of $A$.}
\label{fig:g-fan}
\end{figure}
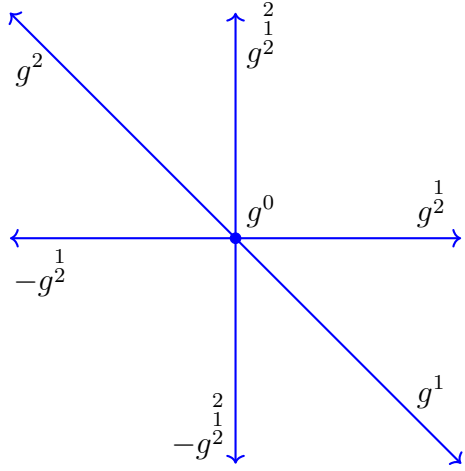
In this case we have that all the TF-equivalence classes in the wall-and-chamber structure of $A$ are of the form $\mathcal{C}_{(M,P)}$ for some $\tau$-rigid pair $(M,P)$ in $\mc A$.

The objects of $\cat{A}$ are as follows: 

\begin{align*} 
& U = \bigl[\mathcal{C}_{(\rep{0},\rep{0})}\bigr],
V = \bigl[\mathcal{C}_{(\rep{1},\rep{0})}\bigr],
W = \bigl[\mathcal{C}_{(\rep{0},\rep{2})}\bigr],\\ & 
X = \bigl[\mathcal{C}_{\bigl(\rep{1\\2},\,\rep{0}\bigr)}\bigr] = \bigl[\mathcal{C}_{\bigl(\rep{0},\,\rep{1\\2}\bigr)}\bigr], 
Y = \bigl[\mathcal{C}_{\Bigl(\rep{2\\1\\2},\,\rep{0}\Bigr)}\bigr] = \bigl[\mathcal{C}_{\Bigl(\rep{0},\,\rep{2\\1\\2}\Bigr)}\bigr], \\ & Z = 
\bigl[\mathcal{C}_{\Bigl(\rep{1},\,\rep{2\\1\\2}\Bigr)}\bigr] = \bigl[\mathcal{C}_{\bigl(\rep{1} \oplus \rep{1 \\2},\, \rep{0}\bigr)}\bigr] =\bigl[\mathcal{C}_{\bigl(\rep{2},\, \rep{1\\2}\bigr)}\bigr] =\bigl[\mathcal{C
}_{\bigl(\rep{2} \oplus \rep{2\\1\\2},\, \rep{0}\bigr)}\bigr] =\bigl[\mathcal{C}_{\bigl(\rep{1\\2} \oplus \rep{2\\1\\2},\, \rep{0}\bigr)}\bigr] =
 \bigl[\mathcal{C}_{\Bigl(\rep{0},\, \rep{1\\2} \oplus \rep{2\\1\\2}\Bigr)}\bigr].
\end{align*}

Let us study the Hom sets $\Hom_{\cat{A}}(U,X)$ and $\Hom_{\cat{A}}(X,Z)$ in more detail. 
By definition, we have that \[\Hom_{\cat{A}}(U,X)  = \biggl\{ f_{\mathcal{C}_{(\rep{0},\rep{0})} \mathcal{C}_{\bigl(\rep{1\\2},\,\rep{0}\bigr)}}, f_{\mathcal{C}_{(\rep{0},\rep{0})} \mathcal{C}_{\bigl(\rep{0},\,\rep{1\\2}\bigr)}} \biggr\} / \sim.\] Since $B(\rep{0},\rep{0}) = A$ and, as we noted in Subsection~\ref{sect:red_wcs}, $\pi_{(\rep{0},\rep{0})}$ restricts to a bijection of the TF-equivalance classes in $N_{(\rep{0},\rep{0})} = \mathbb{R}^n$ and TF-equivalence classes of $\mc A$ in $\mathbb{R}^n$ we conclude that $f_{\mathcal{C}_{(\rep{0},\rep{0})} E'}= f_{\mathcal{C}_{(\rep{0},\rep{0})} E} $ if and only if $E = E'$. Thus, \[\Hom_{\cat{A}}(U,X)  = \biggl\{ \bigl[ f_{\mathcal{C}_{(\rep{0},\rep{0})} \mathcal{C}_{\bigl(\rep{1\\2},\,\rep{0}\bigr)}} \bigr], \bigl[ f_{\mathcal{C}_{(\rep{0},\rep{0})} \mathcal{C}_{\bigl(\rep{0},\,\rep{1\\2}\bigr)}} \bigr] \biggr\}.\]
Now let us consider $\Hom_{\cat{A}}(X, Z)$, which is the set
\[ \biggl\{ 
  f_{ \mathcal{C}_{\bigl(\rep{1\\2},\,\rep{0}\bigr)}\mathcal{C}_{\bigl(\rep{1\\2} \oplus \rep{2\\1\\2},\,\rep{0}\bigr)} }, 
 f_{ \mathcal{C}_{\bigl(\rep{1\\2},\,\rep{0}\bigr)} \mathcal{C}_{\bigl(\rep{1} \oplus \rep{1\\2},\,\rep{0}\bigr)} },
  f_{ \mathcal{C}_{\bigl(\rep{0},\,\rep{1\\2}\bigr)} \mathcal{C}_{\bigl(\rep{2} ,\,\rep{1\\2}\bigr)} },
  f_{ \mathcal{C}_{\bigl(\rep{0},\,\rep{1\\2}\bigr)} \mathcal{C}_{\Bigl(0 ,\,\rep{1\\2}\oplus \rep{2\\1\\2} \Bigr)} }
\biggr\} / \sim.\]
First observe that $\lspan{ \mathcal{C}_{\bigl(\rep{1\\2},\,\rep{0}\bigr)}} = \lspan{(0,-1)}$ and   $\lspan{ \mathcal{C}_{\bigl(\rep{0},\, \rep{1\\2}\bigr)}} = \lspan{(0,1)}$. Thus, for $(x,y) \in \mathbb{R}^2$, $\nu_{ \mathcal{C}_{\bigl(\rep{1\\2},\,\rep{0}\bigr)}} (x,y) = (x,0) = \nu_{ \mathcal{C}_{\bigl(\rep{0},\,\rep{1\\2}\bigr)}} (x,y)$. 
We also compute 
\begin{align*}
\mathcal{C}_{\bigl(\rep{1\\2} \oplus \rep{2\\1\\2},\,\rep{0}\bigr)} & = \{(x, y) \in \mathbb{R}^2 \st x, y<0 \},  \\
\mathcal{C}_{\bigl(\rep{1} \oplus \rep{1\\2},\,\rep{0}\bigr)} & = \{(x, y) \in \mathbb{R}^2 \st  y<0, 0<x< -y \}, \\ 
\mathcal{C}_{\bigl(\rep{2} ,\,\rep{1\\2}\bigr)} & = \{(x, y) \in \mathbb{R}^2 \st  y>0, -y<x<0 \}, \text{ and}  \\
\mathcal{C}_{\bigl(0 ,\,\rep{1\\2}\oplus \rep{2\\1\\2} \bigr) }& = \{(x, y) \in \mathbb{R}^2 \st x, y>0 \}. 
\end{align*}
Together, we see that 
\begin{align*}
\nu_{ \mathcal{C}_{\bigl(\rep{1\\2},\,\rep{0}\bigr)}} \mathcal{C}_{\bigl(\rep{1\\2} \oplus \rep{2\\1\\2},\,\rep{0}\bigr)} &= \{(x,0) \in \mathbb{R}^2 \st  x<0 \}   = \nu_{ \mathcal{C}_{\bigl(\rep{0},\,\rep{1\\2}\bigr)}} \mathcal{C}_{\bigl(\rep{2} ,\,\rep{1\\2}\bigr)} \text{ and} \\
\nu_{ \mathcal{C}_{\bigl(\rep{1\\2},\,\rep{0}\bigr)}} \mathcal{C}_{\bigl(\rep{1} \oplus \rep{1\\2},\,\rep{0}\bigr)} &= \{(x,0) \in \mathbb{R}^2 \st  x>0 \}   = \nu_{ \mathcal{C}_{\bigl(\rep{0},\,\rep{1\\2}\bigr)}} \mathcal{C}_{\Bigl(0 ,\,\rep{1\\2}\oplus \rep{2\\1\\2} \Bigr) }.
\end{align*}
Hence, we have that 
\begin{align*}
\Hom_{\cat{A}}(U,X)=  \biggl\{ 
 &\bigl[ f_{ \mathcal{C}_{\bigl(\rep{1\\2},\,\rep{0}\bigr)}\mathcal{C}_{\bigl(\rep{1\\2} \oplus \rep{2\\1\\2},\,\rep{0}\bigr)} } \bigr] = \bigl[ f_{ \mathcal{C}_{\bigl(\rep{0},\,\rep{1\\2}\bigr)} \mathcal{C}_{\bigl(\rep{2} ,\,\rep{1\\2}\bigr)} } \bigr], \\ 
&\bigl[ f_{ \mathcal{C}_{\bigl(\rep{1\\2},\,\rep{0}\bigr)} \mathcal{C}_{\bigl(\rep{1} \oplus \rep{1\\2},\,\rep{0}\bigr)} } \bigr] = 
\bigl[  f_{ \mathcal{C}_{\bigl(\rep{0},\,\rep{1\\2}\bigr)} 
\mathcal{C}_{\Bigl(0 ,\,\rep{1\\2}\oplus \rep{2\\1\\2} \Bigr)} 
}\bigr]
\biggr\}.
\end{align*}
We do not compute the rest of the category $\cat{A}$ here. In Example~\ref{example2} we compute an equivalent category.
\end{example}

\section{Relation with the $\tau$-cluster morphism category}\label{sect:rel_w_tcm}

 We now show that the category $\cat{A}$ that we defined in the previous section is equivalent to the $\tau$-cluster morphism category $\tcm{A}$. In order to do this, we first define the following poset, which we also view as a category, just as with $TF_A$.

\begin{definition}
Let $\trc{A}$ be the poset whose objects are basic $\tau$-rigid pairs over $A$, with $(M, P) \leqslant (N, Q)$ if $M$ is a direct summand of $N$ and $P$ is a direct summand of $Q$. In this case, we write $h_{(M, P)}^{(N, Q)}$ for the unique morphism which exists from $(M, P)$ to $(N, Q)$.
\end{definition}

In a similar way to how we proceeded in the previous section, we may define a quotient of this category as follows.

\begin{definition}
Let $\qtrc{A}$ be the category defined as follows.
\begin{enumerate}[label = (\Alph*), wide]
\item The objects of $\qtrc{A}$ are equivalence classes of objects of $\trc{A}$ under the equivalence relation where $(M, P) \sim (N, Q)$ if and only if $\tps{}{M, P} = \tps{}{N, Q}$.
\item The morphisms $\Hom_{\qtrc{A}}\big([(M, P)], [(N, Q)]\big)$ consist of \[\bigcup_{\substack{(M', P') \in [(M, P)]\\ (N', Q') \in [(N, Q)]}} \Hom_{\trc{A}}((M', P'), (N', Q'))\] under the equivalence relation where \[h_{(M, P)}^{(M \oplus \widehat{M}, P \oplus \widehat{P})} \sim h_{(M', P')}^{(M' \oplus \widehat{M}', P' \oplus \widehat{P}')}\] if and only if \[\mathcal{E}_{(M, P)}^{\mc A}(\widehat{M}, \widehat{P}) = \mathcal{E}_{(M', P')}^{\mc A}(\widehat{M}', \widehat{P}'),\] noting that $\tps{}{M, P} = \tps{}{M', P'}$ due to the context.
\item The composition of $[h_{(M, P)}^{(M \oplus M', P \oplus P')}]$ and $[h_{(M \oplus M', P \oplus P')}^{(M \oplus M' \oplus M'', P \oplus P' \oplus P'')}]$ is defined to be $[h_{(M, P)}^{(M \oplus M' \oplus M'', P \oplus P' \oplus P'')}]$.
\end{enumerate}
\end{definition}

\begin{example} \label{example2}
As in Example~\ref{example1}, we consider the quiver $Q$ \[\begin{tikzcd}
1 \arrow[r, "\alpha", bend left] & 2 \arrow[l, "\beta", bend left]
\end{tikzcd}\] and the algebra $B = KQ/ \langle \beta \alpha \rangle$.
 Figure~\ref{figure:poset of B} shows Hasse quiver of the poset $\mathfrak{T}(A)$ and in Figure~\ref{fig:qtrcB} we show the category $\qtrc{A}$. In that diagram, non-black arrows with the same label (or colour)  are in the same equivalence class of morphisms. Morphisms from the initial object, $I = {[(\rep{0},\rep{0})]}$ to the terminal object, 
$T$, are obtained by concatenation of arrows under the equivalence relation that $(I \to X \to T) \sim (I \to Y \to T)$ 
if and only if the head of the arrows of $X \to T$ and $Y \to T$ point at the same representative of the equivalence class $T$. 
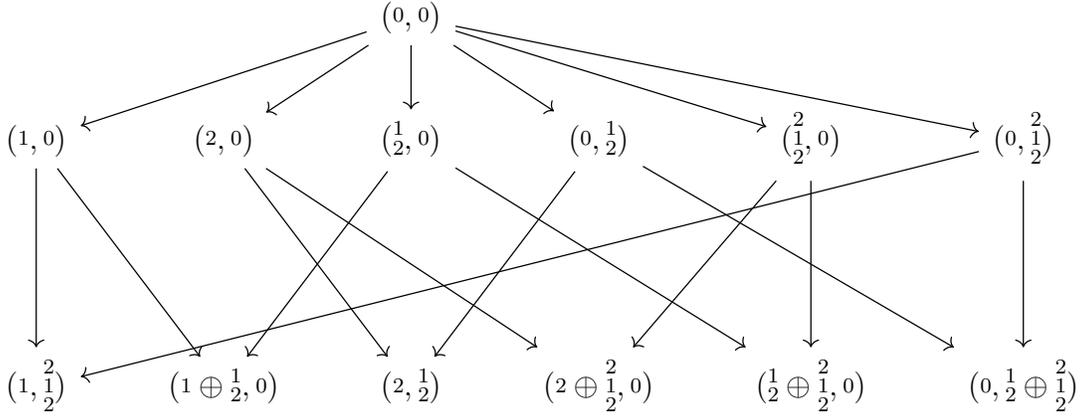
\begin{figure} 
\[\begin{tikzcd}
                                      &                                          & {(\rep{0},\rep{0})} \arrow[lld] \arrow[ld] \arrow[d] \arrow[rd] \arrow[rrd] \arrow[rrrd] &                                             &                                              &                                                  \\
{(\rep{1},\rep{0})} \arrow[ddd] \arrow[rddd] & {(\rep{2},\rep{0})} \arrow[rddd] \arrow[rrddd] & {(\rep{1\\2},\rep{0})} \arrow[rrddd] \arrow[lddd]                                  & {(\rep{0},\rep{1\\2})} \arrow[rrddd] \arrow[lddd] & {(\rep{2\\1\\2},\rep{0})} \arrow[ddd] \arrow[lddd] & {(\rep{0},\rep{2\\1\\2})} \arrow[ddd] \arrow[lllllddd] \\
                                       &                                          &                                                                              &                                             &                                              &                                                  \\
                                       &                                          &                                                                              &                                             &                                              &                                                  \\
{(\rep{1},\rep{2\\1\\2})}            & {(\rep{1}\oplus \rep{1\\2} ,\rep{0})}          & {(\rep{2},\rep{1\\2})}                                                       & {(\rep{2}\oplus \rep{2\\1\\2},\rep{0})}           & {(\rep{1\\2}\oplus \rep{2\\1\\2},\rep{0})}         & {(\rep{0},\rep{1\\2}\oplus \rep{2\\1\\2})}            
\end{tikzcd}\]
\caption{The Hasse quiver of  $\mathfrak{T}(A)$.}
\label{figure:poset of B}
\end{figure}

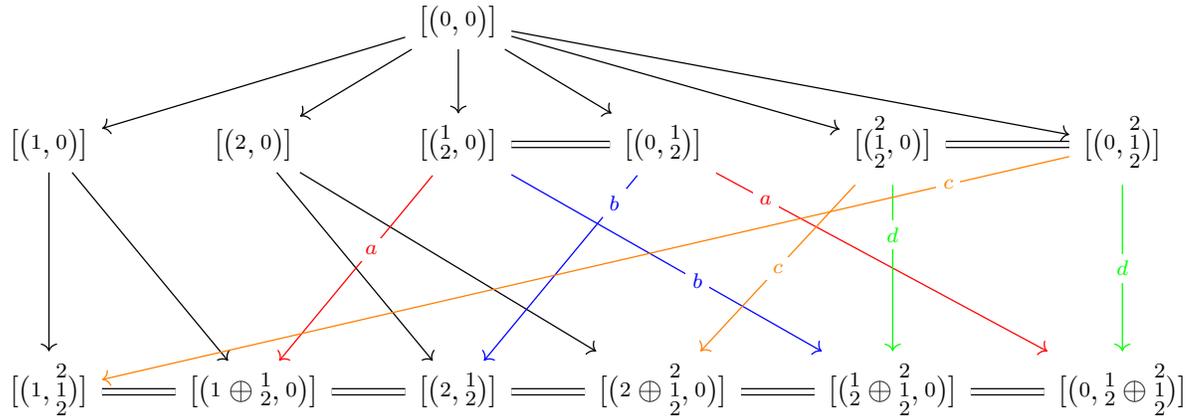
\begin{figure}
\[ \begin{tikzcd}[labels = very near start, labels = description]
                                                             &                                                                  & {[(\rep{0},\rep{0})]} \arrow[lld] \arrow[ld] \arrow[d] \arrow[rd] \arrow[rrd] \arrow[rrrd] &                                                                   &                                                                             &                                                                           \\
{[(\rep{1},\rep{0})]} \arrow[ddd] \arrow[rddd]               & {[(\rep{2},\rep{0})]} \arrow[rddd] \arrow[rrddd]                 & {[(\rep{1\\2},\rep{0})]} \arrow[rrddd, color = blue, "b" {description, pos=0.6}] \arrow[lddd, color = red, "a" {description, pos=0.4}] \arrow[r, equal]                  & {[(\rep{0},\rep{1\\2})]} \arrow[rrddd, color=red, "a" {description, pos=0.15}] \arrow[lddd, color = blue, "b" {description, pos=0.15}]                & {[(\rep{2\\1\\2},\rep{0})]} \arrow[ddd, color=green, , "d" {description, pos=0.3}] \arrow[lddd, color = orange, "c" {description, pos=0.5}] \arrow[r, equal]  & {[(\rep{0},\rep{2\\1\\2})]} \arrow[ddd, color=  green, "d" {description, pos=0.5}] \arrow[lllllddd, color = orange, "c"' description, near start] \\
                                                             &                                                                  &                                                                                            &                                                                   &                                                                             &                                                                           \\
                                                             &                                                                  &                                                                                            &                                                                   &                                                                             &                                                                           \\
{[(\rep{1},\rep{2\\1\\2})]} \arrow[r, equal] & {[(\rep{1}\oplus \rep{1\\2} ,\rep{0})]} \arrow[r, equal] & {[(\rep{2},\rep{1\\2})]} \arrow[r, equal]                                    & {[(\rep{2}\oplus \rep{2\\1\\2},\rep{0})]} \arrow[r, equal] & {[(\rep{1\\2}\oplus \rep{2\\1\\2},\rep{0})]} \arrow[r, equal] & {[(\rep{0},\rep{1\\2}\oplus \rep{2\\1\\2})]}                             
\end{tikzcd}\]

\caption[]{The category $\qtrc{A}$. }  
\label{fig:qtrcB}

\end{figure}

\end{example}

Instead of showing directly that the category $\qtrc{A}$ is well-defined, we show this by showing that it is equivalent to the $\tau$-cluster morphism category $\tcm{A}$.

\begin{proposition}
The $\tau$-cluster morphism category $\tcm{A}$ is equivalent to the category $\qtrc{A}$.
\end{proposition}
\begin{proof}
We define a functor $F \colon \qtrc{A} \to \tcm{A}$. On objects, $F$ sends the equivalence class $[(M, P)]$ to $\tps{}{M, P}$. The equivalence relation ensures that this is well-defined. On morphisms, 
\[F \colon [h_{(M, P)}^{(M \oplus \widehat{M}, P \oplus \widehat{P})}] \longmapsto g^{\mathcal{W}}_{(N, Q)}\] 
where $\mathcal{W} = \tps{}{M, P}$ and $(N, Q) = \mathcal{E}^{\mc A}_{(M, P)}(\widehat{M}, \widehat{P})$. Again, the equivalence relation on morphisms ensures that this is well-defined.

We show that $F$ respects composition. Here we take composable morphisms 
\[[h_{(M, P)}^{(M \oplus M', P \oplus P')}] \text{ and } [h_{(M \oplus M', P \oplus P')}^{(M \oplus M' \oplus M'', P \oplus P' \oplus P'')}]\] 
in $\qtrc{A}$. We must show that the composition of the images of these morphisms under $F$ is equal to the image of $[h_{(M, P)}^{(M \oplus M' \oplus M'', P \oplus P' \oplus P'')}]$, their composition in $\qtrc{A}$. 
We have that
\[F[h_{(M, P)}^{(M \oplus M', P \oplus P')}] = g_{(N', Q')}^{\mathcal{W}}\]
 where $\mathcal{W} = \tps{}{M, P}$ and $(N', Q') = \mathcal{E}^{\mc A}_{(M, P)}(M', P')$; \and \[F[h_{(M \oplus M', P \oplus P')}^{(M \oplus M' \oplus M'', P \oplus P' \oplus P'')}] = g_{(N'', Q'')}^{\mathcal{W}'}\]  where $\mathcal{W}' = \tps{}{M \oplus M', P \oplus P'}$ and $(N'', Q'') = \mathcal{E}^{\mc A}_{(M \oplus M', P \oplus P')}(M'', P'')$. 
 Since we also have $\mathcal{W}' = \tps{\mathcal{W}}{N', Q'}$ by \cite[Theorem~6.4]{BH21}, which generalises \cite[Theorem~4.3]{BM21_wide}, we have that these two morphisms
\begin{align*}
g_{(N', Q')}^{\mathcal{W}}&\colon \mathcal{W} \to \mathcal{W}' \\
g_{(N'', Q'')}^{\mathcal{W}'}&\colon \mathcal{W}' \to \tps{\mathcal{W}'}{N'', Q''}
\end{align*}
are indeed composable. Then, letting $(\widetilde{N''}, \widetilde{Q''}) = \left( \mathcal{E}^{\mathcal{W}}_{(N', Q')} \right)^{-1}(N'', Q'')$, we have that the composition of these two morphisms is $g_{(N' \oplus \widetilde{N''}, Q' \oplus \widetilde{Q''})}^{\mathcal{W}}$, since, again by \cite[Theorem~6.4]{BH21}, we have that $\tps{\mathcal{W}'}{N'' \oplus Q''} = \tps{\mathcal{W}}{(N' \oplus \widetilde{N''}, Q' \oplus \widetilde{Q''}}$. 
But then we have precisely that 
\[F[h_{(M, P)}^{(M \oplus M' \oplus M'', P \oplus P' \oplus P'')}] = g_{(N' \oplus \widetilde{N''}, Q' \oplus \widetilde{Q''})}^{\mathcal{W}},\] 
since $\mathcal{W} = \tps{}{M, P}$ and $(N' \oplus \widetilde{N''}, Q' \oplus \widetilde{Q''}) = \mathcal{E}^{\mc A}_{(M, P)}(M' \oplus M'', P' \oplus P'')$. This is because $(N', Q') = \mathcal{E}^{\mc A}_{(M, P)}(M', P')$ and
\begin{align*}
(\widetilde{N''}, \widetilde{Q''}) &=  \left( \mathcal{E}^{\mathcal{W}}_{(N', Q')} \right)^{-1}(N'', Q'') \\
&= \left( \mathcal{E}^{\mathcal{W}}_{(N', Q')} \right)^{-1}\mathcal{E}^{\mc A}_{(M \oplus M', P \oplus P')}(M'', P'') \\
&= \left( \mathcal{E}^{\mathcal{W}}_{(N', Q')} \right)^{-1}\mathcal{E}^{\mathcal{W}}_{(N', Q')} \mathcal{E}^{\mc A}_{(M, P)}(M'', P'') \\
&= \mathcal{E}^{\mc A}_{(M, P)}(M'', P'').
\end{align*}
Here the penultimate step follows from \cite[Theorem~5.9]{BM21_wide} or \cite[Theorem~6.12]{BH21}.

It is clear that $F$ is essentially surjective, since every $\tau$-perpendicular category emerges from a $\tau$-rigid object by definition. It is likewise clear that $F$ is full, since the $\mathcal{E}$ maps are bijections. Hence $F$ is an equivalence of categories, as desired.
\end{proof}

\begin{theorem}
The category $\qtrc{A}$ is equivalent to the category $\cat{A}$ defined from the wall-and-chamber structure.
\end{theorem}
\begin{proof}
We define a functor $G$ from $\qtrc{A}$ by sending $[(M, P)]$ to $\mathcal{C}_{(M, P)}$ and $[ h_{(M, P)}^{(M \oplus \widehat{M}, P \oplus \widehat{P})} ]$ to $[f_{\mathcal{C}_{(M, P)}\mathcal{C}_{(M \oplus \widehat{M}, P \oplus \widehat{P})}}]$.

We first show that the functor $G$ is well-defined on objects. We have that $[(M, P)] = [(M', P')]$ if and only if $\tps{}{M, P} = \tps{}{M', P'}$. Moreover, we have that $\mathcal{W}_{\mathcal{C}(M, P)} = \tps{}{M, P}$ and that $\mathcal{C}_{(M, P)} \sim \mathcal{C}_{(M', P')}$ if and only if $\mathcal{W}_{\mathcal{C}_{(M, P)}} = \mathcal{W}_{\mathcal{C}_{(M', P')}}$. Consequently, $G$ is well-defined on the objects $[(M, P)]$ of $\qtrc{A}$, since it gives equivalent TF-equivalence classes no matter which equivalence-class representative one chooses in $[(M, P)]$.

We now show that the functor $G$ is well-defined on morphisms. We have that \[[ h_{(M, P)}^{(M \oplus \widehat{M}, P \oplus \widehat{P})} ] = [ h_{(M, P)}^{(M \oplus \widehat{M'}, P \oplus \widehat{P}')} ]\] if and only if \[\mathcal{E}_{(M, P)}^{\mc A}(\widehat{M}, \widehat{P}) = \mathcal{E}_{(M, P)}^{\mc A}(\widehat{M'}, \widehat{P'}).\] We have that \[[f_{\mathcal{C}_{(M, P)}\mathcal{C}_{(M \oplus \widehat{M}, P \oplus \widehat{P})}}] = [f_{\mathcal{C}_{(M, P)}\mathcal{C}_{(M \oplus \widehat{M'}, P \oplus \widehat{P'})}}]\] if and only if \[\np{\mathcal{C}_{(M, P)}}(\mathcal{C}_{(M \oplus \widehat{M}, P \oplus \widehat{P})}) = \np{\mathcal{C}_{(M, P)}}(\mathcal{C}_{(M \oplus \widehat{M'}, P \oplus \widehat{P'})}).\] By Lemma~\ref{lem:proj=red}, we have that this is the case if and only if \[\pi_{\mathcal{C}_{(M, P)}}(\mathcal{C}_{(M \oplus \widehat{M}, P \oplus \widehat{P})}) = \pi_{\mathcal{C}_{(M, P)}}(\mathcal{C}_{(M \oplus \widehat{M'}, P \oplus \widehat{P'})}).\] By \cite[Lemma~4.4]{Asai2021}, we have that this is the case if and only if \[\mathcal{E}_{(M, P)}^{\mc A}(\widehat{M}, \widehat{P}) = \mathcal{E}_{(M, P)}^{\mc A}(\widehat{M'}, \widehat{P'}),\] as desired. This also shows that the functor $G$ is faithful.

The functor $G$ is essentially surjective by construction, since every TF-equivalence class is of the form $\mathcal{C}_{(M, P)}$ for some $\tau$-rigid pair $(M, P)$. 
The functor $G$ is moreover full, since the TF-equivalence classes giving morphisms in $\cat{A}$ are cones in $TF_{A}$, which all arise from $\tau$-rigid pairs $(M, P)$.
Hence, the functor $G$ is an equivalence of categories.
\end{proof}

\begin{corollary}\label{cor:tcm}
The category $\cat{A}$ defined from the wall-and-chamber structure is equivalent to the $\tau$-cluster morphism category $\tcm{A}$.
\end{corollary}

\newcommand{\etalchar}[1]{$^{#1}$}

\end{document}